\documentclass[11pt,reqno]{amsart}
\usepackage{amsmath,amsxtra,latexsym,amsthm,amssymb,amscd,pb-diagram}
\usepackage{mathrsfs,mathabx,amsfonts}
\usepackage[colorlinks=true,linkcolor=red,citecolor=blue]{hyperref}
\usepackage[margin=1in]{geometry}

\usepackage{bm}
\usepackage{color}
\usepackage{makecell,multirow,diagbox}
\usepackage{mathtools}
\usepackage[displaymath, mathlines]{lineno}

\usepackage{graphicx}
\usepackage{epsfig}
\usepackage{array}
\usepackage{enumitem}

\newtheorem{theorem}{Theorem}[section]
\newtheorem{proposition}{Proposition}[section]
\newtheorem{lemma}{Lemma}[section]

\newtheorem{remark}{Remark}[section]

\newcommand{\R}{\mathbb{R}}

\newcommand{\e}{\varepsilon}
\newcommand{\ity}{\infty}

\begin{document}
\title[Semilinear damped wave equations]{Existence of global solutions to semilinear damped wave equations with nonlinearities of derivative type}

\thanks{$^* $\textit{Corresponding author:} Tuan Anh Dao}

\maketitle
\centerline{\scshape \textbf{Dinh Van Duong}}
{\footnotesize
	\centerline{Faculty of Mathematics and Informatics, Hanoi University of Science and Technology}
	\centerline{No.1 Dai Co Viet road, Hanoi, Vietnam}
	\centerline{Email: vanmath2002@gmail.com}}
\medskip

\centerline{\scshape \textbf{Tuan Anh Dao}}
{\footnotesize
	\centerline{Faculty of Mathematics and Informatics, Hanoi University of Science and Technology}
	\centerline{No.1 Dai Co Viet road, Hanoi, Vietnam}
	\centerline{Email: anh.daotuan@hust.edu.vn}}

\begin{abstract}
    In this paper, we would like to consider the semi-linear damped wave equation with the power nonlinearity of derivative type $|u_t|^p$. The main contribution of this work is to improve the results for global (in time) solution existence in a comparison with the pioneering paper \cite{Matsumura1976} of Matsumura, who first established that the solutions exist globally for $p > 1$ ($n = 1$) and $p \ge 2$ ($n \ge 2$). More precisely, we have extended such a result for any $p > 1$ ($n = 1,2$) and $p > 3/2$ ($n = 3$). Our approach relies on constructing appropriately weighted solution spaces linked to the delicate application of several tools from Harmonic Analysis and Banach fixed-point theorem. 
\end{abstract}

\tableofcontents

\section{Introduction}
In this work, let us consider the Cauchy problem for semi-linear damped wave equations as follows:
\begin{equation} \label{Main.Eq.1}
\begin{cases}
u_{tt}(t,x) -\Delta u(t,x) + u_t(t,x)= |u_t(t,x)|^p, &\quad x\in \R^n,\, t > 0, \\
u(0,x) = \varepsilon u_0(x),\quad u_t(0,x) = \varepsilon u_1(x), &\quad x\in \R^n, \\
\end{cases}
\end{equation}
where $p > 1$ stands for the power exponent of nonlinearities and $\varepsilon > 0$ describes the size of initial data. The quantity $|u_t(t,x)|^p$ is called the power nonlinearity of derivative type. To understand our main motivation to investigate this problem more precisely, let us make a brief overview about the research history of the following Cauchy problem:
\begin{equation} \label{Main.Eq.2}
\begin{cases}
u_{tt}(t,x)-\Delta u(t,x)+ u_t(t,x)= F\big(u(t,x)\big), &\quad x\in \mathbb{R}^n,\, t > 0, \\
u(0,x) = \varepsilon u_0(x),\quad u_t(0,x) = \varepsilon u_1(x), &\quad x\in \R^n,
\end{cases}
\end{equation}
where $F(u)= F(u(t,x))$ presents a nonlinear function. Concerning (\ref{Main.Eq.2}) with $F(u) = 0$, well-known as the linear damped wave equations, Matsumura in \cite{Matsumura1976} was the first author to obtain some basic decay estimates. As a sequence, it has been established that the damped wave equation has a diffusive structure as $t \to \infty$. For this reason, several papers since then gave sharp $L^{q_1}-L^{q_2}$ estimates for solutions with $1 \leq q_1 \leq q_2 \leq \infty$, for example, \cite{Nishihara2003,DabbiccoEbert2017,Ikeda2019} and the references therein. Based on the study of the corresponding linear problem, it has attracted the interest of many mathematicians over the years in terms of taking into account the semi-linear problem (\ref{Main.Eq.2}) with $F(u) \ne 0$. To be specific, the authors in \cite{IkehataMiyaokaNakatake2004,IkehataOhta2002} focused on the nonlinear term like $F(u) = u|u|^{p-1}$, moreover, the authors in \cite{Todorova2001,IkehataTanizawa2005} studied the usual power nonlinearity $F(u) = |u|^p$. Among other things, one sees that under the additional regularity $L^1$ for the initial data, the authors in the cited papers showed that the Fujita exponent $p_{\rm Fuj}(n):= 1+2/n$ is the critical exponent of these problems by indicating that the small data solutions exist globally when $p > p_{\rm Fuj}(n)$, and the solutions blow-up in finite time when $p < p_{\rm Fuj}(n)$. As a remaining case, Zhang in \cite{Zhang2001} determined that the value $p= p_{\rm Fuj}(n)$ belongs to the blow-up range for $F(u) = |u|^p$. From these observations, the critical exponent $p_{\rm Fuj}(n)$ is understood as a threshold between the global (in time) existence of small data solutions and the blow-up of solutions even for small data. Recently, Ebert-Girardi-Reissig in \cite{EbertGirardiReissig2020} explored \eqref{Main.Eq.2} with the more general nonlinear term $F(u) = |u|^{p_{\rm Fuj}(n)} \mu(|u|)$, where $\mu=\mu(u)$ is a modulus of continuity, and found out the sharp condition on $\mu(s)$ that separates the global existence case ($I_\mu< \ity$, the called \textit{Dini condition}) and the blow-up case ($I_\mu= \ity$, the so-called \textit{non-Dini condition}). Here the integral quantity $I_\mu$ of a modulus of continuity is given by
\begin{align*}
I_\mu:= \int_{0}^{c}\frac{\mu(s)}{s}\;ds
\end{align*}
with a sufficiently small constant $c>0$. Under the additional regularity $L^m$ for the initial data, with $m \in (1,2]$, the authors in \cite{IkehataOhta2002,DabbiccoEbert2017} and the references therein identified the critical exponent for \eqref{Main.Eq.2} as $p_{\rm crit}(n,m): = p_{\rm Fuj}(n/m) =  1 + 2m/n.$ However, they didn't provide any conclusion regarding the solution's properties when $p = p_{\rm crit}(n,m)$. Afterwards, Ikeda-Inui-Okamoto-Wakasugi in \cite{Ikeda2019} proved the existence of global mild solutions in this critical case for all $m \in (1,2]$ in low dimensional spaces along with some additional conditions. \medskip

Coming back to our main problem \eqref{Main.Eq.1}, we can recognize that the first paper \cite{Matsumura1976} from 1976 of Matsumura successfully proved the global (in time) existence of solutions for the entire range $p > 1$ ($n = 1$), but for the required condition $p \ge 2$ ($n \ge 2$). It is quite surprising that since then there seem not so many research papers concerning an improvement of \cite{Matsumura1976} in terms of the study of \eqref{Main.Eq.1} to the best of the author's knowledge. From this fact, we aim in this work to establish the bigger range for existence of global (in time) solutions to (\ref{Main.Eq.1}) depending on some dimensions, particularly, for all $p > 1$ ($n = 1,2$) and for $p > 3/2$ ($n = 3$). To do this, we need to construct the appropriately weighted solution spaces and effectively apply several key tools from Harmonic Analysis associated with the Banach fixed-point theorem. The point worth noticing is that the small data Sobolev solutions to \eqref{Main.Eq.1} exist globally for any $p>1$, which has never appeared in previous literature, at least in the low dimensional cases. Throughout this paper, one also observes that when the usual power nonlinearity $|u|^p$ is replaced by the power nonlinearity of derivative type $|u_t|^p$, the critical exponent shifts to the left in a comparison with the Fujita exponent. \medskip

\textbf{Notations:} We write $f\lesssim g$ when there exists a constant $C>0$ such that $f\leq Cg$, and $f \sim g$ when $g\lesssim f\lesssim g$. For any $\gamma \in \R$, we denote $[\gamma]^+:= \max\{\gamma,0\}$ and $\lceil \gamma \rceil := \min\{k \in \mathbb{Z}: k \geq \gamma\}$. As usual, $H^{a}_q$ and $\dot{H}^{a}_q$, with $q \in (1, \infty)$, $a \geq 0$, denote potential spaces based on $L^q$ spaces. Here $\big< \nabla\big>^{a}$ and $|\nabla|^{a}$ stand for the pseudo-differential operators with symbols $\big<\xi\big>^{a}$ and $|\xi|^{a}$, respectively, where the symbol $\langle x\rangle := \sqrt{|x|^2 +1} $ denotes the Japanese bracket. \medskip

Let us state the global (in time) existence of small data Sobolev solutions, which will be proved in this paper. \medskip
\begin{theorem}[\textbf{Main result}]\label{Theorem1}
     Let $n \geq 1$ and the power exponent $p$ fulfills the following condition:
    \begin{equation} \label{Condition_p}
         p > \max\left\{1,n/2\right\}.
     \end{equation} 
     Assume that $s \in \left(n/2, p\right)$ and the initial data belong to the class
     \begin{align*}
            (u_0, u_1) \in \mathcal{D}[n,s,\alpha] := (H^{s+1} \cap H^{\delta_\alpha+1}_\alpha \cap L^1)\times (H^{s} \cap H_\alpha^{\beta_{\alpha}} \cap L^1) 
            \end{align*}
            with the norm $\|(u_0, u_1)\|_{\mathcal{D}[n,s,\alpha]}:= \|u_0\|_{H^{s+1}}+ \|u_0\|_{H^{\delta_\alpha+1}_\alpha}+ \|u_0\|_{L^1}+ \|u_1\|_{H^s}+ \|u_1\|_{H^{\beta_\alpha}_\alpha}+ \|u_1\|_{L^1}$, where
            $\alpha := \min\{2, p\}, \,\,
        \beta_{\alpha} := (n-1)\left(1/\alpha-1/2\right), \,\, \delta_{\alpha} := n \left(1/\alpha-1/2\right).$ 
       Then, there exists a constant $\varepsilon_0 > 0$ such that for any $\varepsilon \in (0, \varepsilon_0]$, the Cauchy problem \eqref{Main.Eq.1} admits a unique global (in time) Sobolev solution 
        \begin{align*}
            u \in \mathcal{C}\big([0, \infty), L^2\big) \cap \mathcal{C}^1\big([0, \infty), H^s \cap L^\alpha\big).
        \end{align*}
      Furthermore, the following estimates hold for all $t > 0$:
        \begin{align*}
            \|u_t(t,\cdot)\|_{L^\alpha} &\lesssim \varepsilon (1+t)^{-\frac{n}{2}(1-\frac{1}{\alpha})-1} \|(u_0,u_1)\|_{\mathcal{D}[n,s,\alpha]},\\
            \|u_t(t,\cdot)\|_{\dot{H}^s} &\lesssim \varepsilon(1+t)^{-\frac{n}{4}-1-\frac{s}{2}} \|(u_0,u_1)\|_{\mathcal{D}[n,s,\alpha]}.
        \end{align*}
\end{theorem}

\begin{remark}
\fontshape{n}
\selectfont
    From the statement of Theorem \ref{Theorem1}, it follows that the Cauchy problem \eqref{Main.Eq.1} has a unique global (in time) Sobolev solution for all $p > 1$ in low dimensional spaces $n=1,2$. Moreover, the achieved condition \eqref{Condition_p} for $p$ tells us that the critical exponent shifts to the left in a comparison with the well-known Fujita exponent $p_{\rm Fuj}(n):= 1+2/n$ in terms of the study of semi-linear damped wave equations with the power nonlinearity of derivative type $|u_t|^p$ instead of the usual power nonlinearity $|u|^p$.
\end{remark}


    \section{Proof of the main result}
\subsection{Preliminaries}

To begin with, we can write the solutions to the corresponding linear problem of $(\ref{Main.Eq.2})$, i.e. with the vanishing right-hand side $F(u) \equiv 0$, by the formula
\begin{equation*}
u^{\rm lin}(t,x) = \varepsilon \big(\mathcal{K}(t,x) +\partial_t \mathcal{K}(t,x)\big)\ast_x u_0(x) + \varepsilon \mathcal{K}(t,x) \ast_x u_1(x),
\end{equation*}
where the Fourier transform of the kernel $\mathcal{K}(t,x)$ is defined by
\begin{align*}
    \widehat{\mathcal{K}}(t,\xi) = 
\displaystyle\frac{e^{-\frac{t}{2}}\sinh{\left(t \sqrt{\frac{1}{4} -|\xi|^2}\right)}}{\sqrt{\frac{1}{4}- |\xi|^2}}  \text{ if } |\xi| \leq \displaystyle\frac{1}{2}
\,\,\text{ and }\,\, \widehat{\mathcal{K}}(t, \xi) = \displaystyle\frac{e^{-\frac{t}{2}}\sin{\left(t \sqrt{|\xi|^2-\frac{1}{4}}\right)}}{\sqrt{|\xi|^2-\frac{1}{4}}} \text{ if } |\xi| > \displaystyle\frac{1}{2}.
\end{align*}
 So, thanks to Duhamel's principle, the solutions to $(\ref{Main.Eq.1})$ can be expressed by
\begin{align*}
    u(t,x)
    &= u^{\rm lin}(t,x) + u^{\rm non}(t,x), \quad\text{ where }\quad u^{\rm non}(t,x) := \int_0^t \mathcal{K}(t-\tau,x) \ast_x |u_t(\tau,x)|^p d\tau.
\end{align*}
Let $\chi_k= \chi_k(r)$ with $k\in\{\rm L,H\}$ be smooth cut-off functions having the following properties:
\begin{align*}
&\chi_{\rm L}(r)=
\begin{cases}
1 &\quad \text{ if }r\le \varepsilon^*/2, \\
0 &\quad \text{ if }r\ge \varepsilon^*
\end{cases}
\qquad \text{ and } \qquad
\chi_{\rm H}(r)= 1 -\chi_{\rm L}(r),
\end{align*}
where $\varepsilon^*$ is a sufficiently small, positive constant.
It is obvious to see that $\chi_{\rm H}(r)= 1$ if $r \geq \varepsilon^*$ and $\chi_{\rm H}(r)= 0$ if $r \le \varepsilon^*/2$.
Now, we recall the following ingredients for the kernel $\mathcal{K}(t,x)$, which play an essential role in our proof.

\begin{lemma}[Linear estimates]\label{LinearEstimates}
    Let $n \geq 1, \,
    1 \leq r \leq q < \infty$, $q \ne 1$ and $s_1 \geq s_2 \geq 0$.  Then, the following estimates hold for all $t > 0$ (see Theorem 1.1 in \cite{Ikeda2019}):
    \begin{align*}
        \big\| |\nabla|^{s_1} \mathcal{K}(t,x)\ast_x g(x)\big\|_{L^q} &\lesssim (1+t)^{-\frac{n}{2}(\frac{1}{r}-\frac{1}{q})-\frac{s_1-s_2}{2}} \big\||\nabla|^{s_2} \chi_{\rm L}(|\nabla|)g\big\|_{L^{r}}  + e^{-ct} \big\||\nabla|^{s_1}\chi_{\rm H}(|\nabla|) g\big\|_{H_{q}^{\beta_q-1}},\\
         \big\||\nabla|^{s_1} \partial_t \mathcal{K}(t,x)\ast_x g(x)\big\|_{L^q} &\lesssim (1+t)^{-\frac{n}{2}(\frac{1}{r}-\frac{1}{q})-\frac{s_1-s_2}{2}-1} \big\||\nabla|^{s_2} \chi_{\rm L}(|\nabla|)g\big\|_{L^{r}}  + e^{-ct} \big\||\nabla|^{s_1}\chi_{\rm H}(|\nabla|) g\big\|_{H_{q}^{\beta_q}},
        \end{align*}
        where $c$ is a suitable positive constant and
       $
            \beta_q := (n-1)|1/2-1/q|.
        $ Moreover, we also obtain the following estimate (see Propositions 4.1 and 4.2 in \cite{DabbiccoEbert2017}):
    \begin{align*}
        \big\||\nabla|^{s_1} \partial_t^2 \mathcal{K}(t,x) \ast_{x} g(x) \big\|_{L^{q}} \lesssim (1+t)^{-\frac{n}{2}(\frac{1}{r}-\frac{1}{q})-\frac{s_1}{2}-2} \|g\|_{L^{r}} + e^{-ct} \|g\|_{H_{q}^{\delta_q+s_1+1}},  
    \end{align*}
    where 
    $
        \delta_q := n |1/2-1/q|.
    $
\end{lemma}
\begin{remark}
\fontshape{n}
\selectfont
The third estimate in Lemma \ref{LinearEstimates} is solely used to determine the required data space in Theorem \ref{Theorem1}. The remaining two estimates will be frequently applied to our proof in the next sequel.
\end{remark}

Under the assumptions of Theorem \ref{Theorem1},  we define the following function spaces for $T>0$:
\begin{align*}
    X(T) :=  L^{\infty}\big([0, T], H^s \cap L^{\alpha}\big)
\end{align*}
with the norm
\begin{align*}
    \|\varphi\|_{X(T)} := &\sup _{t \in [0,T]} \bigg\{(1+t)^{\frac{n}{2}(1-\frac{1}{\alpha})+1}\|\varphi(t,\cdot)\|_{L^{\alpha}} + (1+t)^{\frac{n}{4}+1+\frac{s}{2}} \| \varphi(t,\cdot)\|_{\dot{H}^s} \bigg\}
\end{align*}
and 
\begin{align*}
    Y(T) := L^{\infty}([0,T], L^2)
\end{align*}
with the norm
\begin{align*}
    \|\varphi\|_{Y(T)} := \sup_{t \in [0,T]
    }\big\{(1+t) \|\varphi(t,\cdot)\|_{L^2} \big\}.
\end{align*}
Moreover, we denote $X(T, M) := \{u \in X(T): \|u\|_{X(T)} \leq M \}$ for all $M > 0$ and introduce the function space
\begin{align*}
    Z(T) := \{\varphi \text{ is a first-order differentiable function on $[0,T]$ such that } \varphi_t \in X(T)\}.
\end{align*}
Next, we define the operator $\mathcal{N}$ on the space $Z(T)$ by
\begin{align}\label{Mapping1}
   \mathcal{N}[u](t,x) := u^{\rm lin}(t,x) + u^{\rm non}(t,x).
\end{align}
Following the proof of Lemma A.1 in \cite{Ikeda2019} with some minor modifications, we may conclude the following lemma.
\begin{lemma}\label{Lem2.2}
     For all $M > 0$, $X(T, M)$  is a closed subset of $Y(T)$ with respect to the metric $Y(T)$.
\end{lemma}

\subsection{Auxiliary estimates}
\begin{lemma}\label{Lemma1.2}
    Under the assumptions of Theorem \ref{Theorem1}, the following estimates hold for all $u \in Z(T)$, $\tau > 0$ and $\sigma \geq 1$:
    \begin{align}
        \big\| |u_t(\tau,\cdot)|^p\big\|_{L^{\sigma}} &\lesssim (1+\tau)^{-p-\frac{n}{2}(p-\frac{1}{\sigma})} \|u_t\|_{X(T)}^p,\label{Es.Lemma1.1}\\ 
        \big\||u_t(\tau,\cdot)|^p\big\|_{\dot{H}_{\alpha}^{\beta_\alpha}} &\lesssim (1+\tau)^{-p-\frac{n}{2}(p-\frac{1}{\alpha})-\frac{\beta_\alpha}{2}} \|u_t\|_{X(T)}^p,\label{Es.Lemma1.2}\\
        \big\| |u_t(\tau,\cdot)|^p\big\|_{\dot{H}^s} &\lesssim (1+\tau)^{-p-\frac{n}{2}(p-\frac{1}{2})-\frac{s}{2}+\kappa} \|u_t\|_{X(T)}^p,\label{Es.Lemma1.3}
    \end{align}
    where $\kappa$ is an arbitrarily small positive constant.
\end{lemma}
\begin{proof}
    Using Proposition \ref{fractionalGagliardoNirenberg} we arrive at
    \begin{align*}
        \big\||u_t(\tau,\cdot)|^p\big\|_{L^\sigma} = \|u_t(\tau,\cdot)\|_{L^{p\sigma}}^p\lesssim \|u_t(\tau, \cdot)\|_{L^\alpha}^{p(1-\omega_0)} \|u_t(\tau, \cdot)\|_{\dot{H}^s}^{p\omega_0} \lesssim (1+\tau)^{-p-\frac{n}{2}(p-\frac{1}{\sigma})} \|u_t\|_{X(T)}^p,
    \end{align*}
    where we note that the following condition is satisfied due to $s> n/2$:
    \begin{align*}
       \omega_0 := \frac{\frac{1}{\alpha}-\frac{1}{p\sigma}}{\frac{1}{\alpha}-\frac{1}{2}+\frac{s}{n}} \in [0, 1].
    \end{align*}
   This leads to the estimate (\ref{Es.Lemma1.1}). Next, we prove the estimate (\ref{Es.Lemma1.2}). If $\beta_\alpha =0$, it is estimate (\ref{Es.Lemma1.1}) with $\sigma =\alpha$. Otherwise, we apply Propositions \ref{fractionalGagliardoNirenberg} and \ref{chainrule} with $p > n/2 > \lceil \beta_\alpha \rceil$ to obtain
   \begin{align*}
        \big\||u_t(\tau,\cdot)|^p\big\|_{\dot{H}_{\alpha}^{\beta_\alpha}} &\lesssim \|u_t(\tau,\cdot)\|_{L^{r_1}}^{p-1} \|u_t(\tau,\cdot)\|_{\dot{H}_{r_2}^{\beta_{\alpha}}}\lesssim \|u_t(\tau,\cdot)\|_{L^{\alpha}}^{(p-1)(1-\omega_1)+1-\omega_2} \|u_t(\tau,\cdot)\|_{\dot{H}^s}^{(p-1)\omega_1 +\omega_2}\\
       &\lesssim (1+\tau)^{-p-\frac{n}{2}(p-\frac{1}{\alpha})-\frac{\beta_\alpha}{2}}\|u_t\|_{X(T)}^p,
   \end{align*}
   provided that there exist parameters $r_1, r_2 > 1$ satisfying
   \begin{align*}
       \omega_1 := \dfrac{\frac{1}{\alpha}-\frac{1}{r_1}}{\frac{1}{\alpha}-\frac{1}{2}+\frac{s}{n}} \in [0,1] ; \quad \omega_2 := \dfrac{\frac{1}{\alpha}-\frac{1}{r_2}+\frac{\beta_\alpha}{n}}{\frac{1}{\alpha}-\frac{1}{2}+\frac{s}{n}} \in [0,1] \,\,\text{ and }\,\,
       \dfrac{1}{\alpha} = \dfrac{p-1}{r_1} +\dfrac{1}{r_2}.
   \end{align*}
   The fact is that we can choose
  $
       1/r_1 := \kappa/(p-1) \text{ and } 1/r_2 := 1/\alpha-\kappa
  $
   to verify these conditions by the aid of $\beta_\alpha < n/2 < s$. Therefore, the estimate (\ref{Es.Lemma1.2}) is proved. Finally, thanks to Propositions \ref{fractionalGagliardoNirenberg}, \ref{FractionalPowers} and \ref{Embedding} with 
   \begin{align}
    \frac{n}{2} -\frac{2\kappa}{p-1} =: d < \frac{n}{2} < s <p,\label{Para1}
    \end{align}
    we gain
   \begin{align*}
       \big\||u_t(\tau,\cdot)|^p\big\|_{\dot{H}^s} &\lesssim \|u_t(\tau,\cdot)\|_{\dot{H}^s} \left(\|u_t(\tau,\cdot)\|_{\dot{H}^{d}}+  \|u_t(\tau,\cdot)\|_{\dot{H}^s}\right)^{p-1}\lesssim (1+\tau)^{-p-\frac{n}{2}(p-\frac{1}{2})-\frac{s}{2}+\kappa} \|u_t\|_{X(T)}^p. 
   \end{align*}
   Then, we can conclude the estimate (\ref{Es.Lemma1.3}). This completes the proof of Lemma \ref{Lemma1.2}.
\end{proof}
\begin{proposition}\label{Pro1.1}
     Under the assumptions of Theorem \ref{Theorem1}, the following estimate holds for all $u \in Z(T)$:
     \begin{align*}
         \|\partial_t u^{\rm non}\|_{X(T)} \lesssim \|u_t\|_{X(T)}^p.
     \end{align*}
\end{proposition}
     \begin{proof}
         Using Lemma \ref{Lemma1.2} combined with
         Lemma \ref{LinearEstimates} for $r = 1, q = \alpha$, $s_1 = s_2 =0$ for $\tau \in [0, t/2]$, and $r=q =\alpha, s_1=s_2 =0$ for $\tau \in (t/2, t]$ we derive
         \begin{align}
             \|\partial_t u^{\rm non}(t,\cdot)\|_{L^{\alpha}} &\lesssim \int_0^{t/2} (1+t-\tau)^{-\frac{n}{2}(1-\frac{1}{\alpha})-1} \big\| |u_t(\tau,\cdot)|^p\big\|_{L^1 \cap \dot{H}_\alpha^{\beta_\alpha}} d\tau \notag\\
             &\hspace{3cm}+ \int_{t/2}^t (1+t-\tau)^{-1} \big\||u_t(\tau,\cdot)|^p\big\|_{L^\alpha \cap \dot{H}^{\beta_\alpha}_\alpha} d\tau\notag\\
             &\lesssim (1+t)^{-\frac{n}{2}(1-\frac{1}{\alpha})-1} \|u_t\|_{X(T)}^p \int_0^{t/2} (1+\tau)^{-p-\frac{n}{2}(p-1)} d\tau \notag\\
             &\hspace{3cm}+ (1+t)^{-p-\frac{n}{2}(p-\frac{1}{\alpha})} \log(e+t) \|u_t\|_{X(T)}^p \notag\\
             &\lesssim (1+t)^{-\frac{n}{2}(1-\frac{1}{\alpha})-1} \|u_t\|_{X(T)}^p, \label{Main.Es.1}
         \end{align}
    where we can see that $
        -p-n(p-1)/2 < -1 \text{ for all } p >1.
    $
 Next, thanks to again Lemma \ref{Lemma1.2} combined with
         Lemma \ref{LinearEstimates} for $r = 1, q = 2$, $s_1 =s, s_2 =0$ for $\tau \in [0, t/2]$, and $r=q =2, s_1=s_2 =s$ for $\tau \in (t/2, t]$, we obtain
    \begin{align}
        \|\partial_t u^{\rm non}(t,\cdot)\|_{\dot{H}^s} &\lesssim \int_0^{t/2} (1+t-\tau)^{-\frac{n}{4}-1-\frac{s}{2}} \big\||u_t(\tau,\cdot)|^p\big\|_{L^1 \cap \dot{H}^s} d\tau +\int_{t/2}^t (1+t-\tau)^{-1} \big\||u_t(\tau,\cdot)|^p\big\|_{\dot{H}^s} d\tau\notag\\
        &\lesssim (1+t)^{-\frac{n}{4}-1-\frac{s}{2}} \|u_t\|_{X(T)}^p. \label{Main.Es.2} 
    \end{align}
    Combining both the estimates (\ref{Main.Es.1}) and (\ref{Main.Es.2}), we have completed the proof of Proposition \ref{Pro1.1}.
     \end{proof}

\begin{proposition}\label{Pro1.2}
    Under the assumptions of Theorem \ref{Theorem1}, the following estimate holds for all $u,\,v \in Z(T)$:
    \begin{align*}
        \|\partial_t\mathcal{N}[u]-\partial_t\mathcal{N}[v]\|_{Y(T)} \lesssim \|u_t-v_t\|_{Y(T)}\left(\|u_t\|_{X(T)}^{p-1} + \|v_t\|_{X(T)}^{p-1}\right).
    \end{align*}
\end{proposition}
    \begin{proof}
        From the definition of $Y(T)$ combined with Lemma \ref{LinearEstimates} with $r = m \in (1,2)$, $q =2$ and $s_1=s_2=0$, we obtain
        \begin{align*}
            &\|\partial_t\mathcal{N}[u]-\partial_t\mathcal{N}[v]\|_{Y(T)}  \lesssim \sup_{t \in [0,T]} \left\{(1+t)\int_0^t (1+t-\tau)^{-\frac{n}{2}(\frac{1}{m}-\frac{1}{2})-1}\big\||u_t(\tau,\cdot)|^p-|v_t(\tau,\cdot)|^p\big\|_{L^m \cap L^2} d\tau\right\}.
        \end{align*}
    Thanks to H\"older's inequality and Proposition \ref{fractionalGagliardoNirenberg}, one has
    \begin{align*}
        \big\||u_t(\tau,\cdot)|^p-|v_t(\tau,\cdot)|^p\big\|_{L^m} &\lesssim \|u_t(\tau,\cdot)-v_t(\tau,\cdot)\|_{L^2} \left(\|u_t(\tau,\cdot)\|_{L^{\gamma(p-1)}}^{p-1} + \|u_t(\tau,\cdot)\|_{L^{\gamma(p-1)}}^{p-1}\right)\\
        &\lesssim (1+\tau)^{-p-\frac{n}{2}(p-\frac{1}{m}-\frac{1}{2})} \|u_t-v_t\|_{Y(T)}\left(\|u_t\|_{X(T)}^{p-1} + \|v_t\|_{X(T)}^{p-1}\right),
    \end{align*}
    provided that the following conditions are satisfied:
    \begin{align*}
     m \in (1,2),\quad  \frac{1}{m}=\frac{1}{2}+\frac{1}{\gamma} \quad\text{ and } \quad \frac{\frac{1}{\alpha}-\frac{1}{\gamma(p-1)}}{\frac{1}{\alpha}-\frac{1}{2}+\frac{s}{n}} \in [0, 1].
    \end{align*}
    Therefore, we can choose the parameter $\gamma$ fulfilling
    $0 < 1/\gamma < \min\left\{1/2, (p-1)/\alpha\right\} $
    to verify all above conditions. Next, using again (\ref{Para1}), Propositions \ref{fractionalGagliardoNirenberg} and \ref{Embedding}, one finds
    \begin{align*}
        \|u_t(\tau,\cdot)\|_{L^\infty} &\lesssim \|u_t(\tau,\cdot)\|_{\dot{H}^{d}}+  \|u_t(\tau,\cdot)\|_{\dot{H}^s} \lesssim (1+\tau)^{-1-\frac{n}{2}+\frac{\kappa}{p-1}} \|u_t\|_{X(T)}.
    \end{align*}
    Thus, it follows that
    \begin{align*}
        \big\||u_t(\tau,\cdot)|^p-|v_t(\tau,\cdot)|^p\big\|_{L^2}
        &\lesssim (1+\tau)^{-p-\frac{n}{2}(p-1)+\kappa} \|u_t-v_t\|_{Y(T)} \left(\|u_t\|_{X(T)}^{p-1} + \|v_t\|_{X(T)}^{p-1}\right).
    \end{align*}
    Summarizing, we obtain the following estimate:
    \begin{align*}
        \|\partial_t\mathcal{N}[u]-\partial_t\mathcal{N}[v]\|_{Y(T)} &\lesssim \sup_{t \in [0,T]} \bigg\{(1+t) \int_0^t (1+t-\tau)^{-\frac{n}{2}(\frac{1}{m}-\frac{1}{2})-1}(1+\tau)^{-p-\frac{n}{2}(p-\frac{1}{m}-\frac{1}{2})} d\tau\bigg\} \\
        & \qquad \times\|u_t-v_t\|_{Y(T)}\left(\|u_t\|_{X(T)}^{p-1} + \|v_t\|_{X(T)}^{p-1}\right),\\
        &\lesssim \|u_t-v_t\|_{Y(T)}\left(\|u_t\|_{X(T)}^{p-1} + \|v_t\|_{X(T)}^{p-1}\right).
    \end{align*}
    Here, the fact $m \in (1,2)$ and $p-1>1/\gamma$, it implies
    \begin{align*}
     -p-\frac{n}{2}\left(p-\frac{1}{m}-\frac{1}{2}\right) < -1. 
     \end{align*}
    All in all, we have established the proof of Proposition \ref{Pro1.2}.
 \end{proof}

\subsection{Proof of Theorem \ref{Theorem1}} 
At the first step, using again Lemma \ref{LinearEstimates}, we have $
    \|\partial_t u^{\rm lin}\|_{X(T)} \lesssim \varepsilon\|(u_0, u_1)\|_{\mathcal{D}[n,s,\alpha]}
$
for all $T > 0$. By \eqref{Mapping1}, linking this to Propositions \ref{Pro1.1} and \ref{Pro1.2} we obtain the following estimates for all $u,\,v \in Z(T)$:
    \begin{align}
        \|\partial_t \mathcal{N}[u]\|_{X(T)} &\leq C_1 \varepsilon \|(u_0, u_1)\|_{\mathcal{D}[n,s,\alpha]} + C_1\|u_t\|_{X(T)}^p, \label{Es.Pro2.1}\\
        \|\partial_t\mathcal{N}[u]-\partial_t\mathcal{N}[v]\|_{Y(T)} &\leq C_2\|u_t-v_t\|_{Y(T)}\left(\|u_t\|_{X(T)}^{p-1}+\|v_t\|_{X(T)}^{p-1}\right).\label{Es.Pro2.2}
    \end{align}
    Let us fix 
    $
    M := 2C_1 \|(u_0, u_1)\|_{\mathcal{D}[n,s,\alpha]}
    $
    and the parameter $\varepsilon_0$ such that
    $
        \max\{C_1, 2C_2\}M^{p-1}\varepsilon_0^{p-1} \leq 1/2.
   $
   Next, taking the recurrence sequence $\{u_j\}^\ity_{j=0} \subset Z(T)$ with $u_0 = 0;\,  u_{j} = \mathcal{N}[u_{j-1}]$ for $j \ge 1$, we employ the estimate (\ref{Es.Pro2.1}) to have
   $\|\partial_t u_j\|_{X(T)} \leq M\varepsilon$
   for all $j \in \mathbb{N}$ and $\varepsilon \in (0, \varepsilon_0]$. 
Moreover, using the estimate (\ref{Es.Pro2.2}) one has
\begin{align*}
    \|\partial_t u_{j+1}-\partial_t u_j\|_{Y(T)} &\leq C_2M^{p-1}\varepsilon^{p-1} \|\partial_t u_j- \partial_t u_{j-1}\|_{Y(T)}\leq \frac{1}{2} \|\partial_t u_{j}-\partial_t u_{j-1}\|_{Y(T)},
\end{align*}
so that $\{\partial_t u_j\}^\ity_{j=0} \subset X(T, M\varepsilon)$ is a Cauchy sequence in the Banach
space $Y(T)$. As a consequence, combining this with Lemma \ref{Lem2.2} we claim that there exists a function $\varphi \in X(T, M\varepsilon)$ satisfying
$
     \partial_t u_j \to \varphi \text{ in } X(T, M\varepsilon)
$
as $j \to \infty$ with the  metric $Y(T)$.
Next, let us consider the following function:
\begin{align*}
    u^{*}(t,x) =  u_0(x)+ \int_0^t \varphi(\tau,x) d\tau.
\end{align*}
Additionally, one has
\begin{align*}
    \lim_{t \to t_0} \|u^*(t,\cdot)-u^*(t_0,\cdot)\|_{L^2} \lesssim \lim_{t \to t_0} \left|\int_{t_0}^t \|\varphi(\tau,\cdot)\|_{X(T)} d\tau\right| \leq M\varepsilon\lim_{t \to t_0}  |t-t_0| = 0
\end{align*}
for all $t_0 \in [0,T]$. From this, we obtain
$
    u^* \in \mathcal{C}\big([0, T], L^2\big).
$
Moreover, thanks to again the estimate (\ref{Es.Pro2.2}), it holds 
\begin{align*}
    \|\partial_t u_{j+1} - \partial_t \mathcal{N}[u^*]\|_{Y(T)} &\leq C_2\|\partial_t u_j - \varphi\|_{Y(T)}\Big(\|\partial_t u_j\|_{X(T)}^{p-1}+\|\varphi\|_{X(T)}^{p-1}\Big)\leq \frac{1}{2} \|\partial_t u_j - \varphi\|_{Y(T)},
\end{align*}
that is, 
$
    \partial_t u_j \to \partial_t \mathcal{N}[u^*]
$
as $j \to \infty$ with the  metric $Y(T)$. Therefore, we can immediately conclude
$
    \varphi \equiv \partial_t\mathcal{N}[u^*],
$
that is, there exists a function $\gamma(x)$ satisfies
\begin{align}\label{Eq1}
    u^*(t,x) = \gamma(x)+ \mathcal{N}[u^*](t,x)
\end{align}
for all $t \geq 0$ and $x 
\in \mathbb{R}^n$.
 Taking $t=0$ in (\ref{Eq1}), we can see that $\gamma(x) \equiv 0$ on $\mathbb{R}^n$ due to the fact $u^*(0,x) = u_0(x)$. For this reason, we have a solution $u^* = \mathcal{N}[u^*] \text{ in }\mathcal{C}\big([0,T], L^2\big)$ for all $T >0$. Since $T$ is arbitrary, we can conclude that $u^* \in \mathcal{C}\big([0,\infty), L^2\big)$.
 
In the second step, we need to show that
 $ u^* \in \mathcal{C}^1\big([0,\infty), H^s \cap L^{\alpha}\big), \text{ that is, }\varphi \in \mathcal{C}\big([0,\infty), H^s \cap L^{\alpha}\big). $
 To indicate this property, let us recall the relation
 \begin{align*}
     \varphi(t,x) = \partial_t \mathcal{N}[u^*](t,x)
     &= \partial_t u^{\rm lin}(t,x) +  \int_0^t \partial_t \mathcal{K}(t-\tau,x) \ast_x |\varphi(\tau,x)|^p d\tau.
 \end{align*}
 Since the linear part obviously satisfies the continuous property, it suffices to show that
\begin{align}
    \int_0^t \partial_t \mathcal{K}(t-\tau,x) \ast_x |\varphi(\tau,x)|^p d\tau \in \mathcal{C}\big([0,\infty), H^s \cap L^{\alpha}\big). \label{EQ2}
\end{align}
Noting that $\varphi \in X(\infty, M\varepsilon)$ and using again Lemmas \ref{LinearEstimates}-\ref{Lemma1.2}, we have the following estimates:
\begin{align*}
    \| \partial_t \mathcal{K}(t-\tau,x) \ast_x |\varphi(\tau,x)|^p\|_{L^{\alpha}} &\lesssim (1+\tau)^{-p-\frac{n}{2}(p-\frac{1}{\alpha})} \|\varphi\|_{X(\infty)}^p,\\
    \| \partial_t \mathcal{K}(t-\tau,x) \ast_x |\varphi(\tau,x)|^p\|_{\dot{H}^s} &\lesssim (1+\tau)^{-p-\frac{n}{2}(p-\frac{1}{2})-\frac{s}{2}+\kappa} \|\varphi\|_{X(\infty)}^p,
\end{align*}
where $\kappa$ is an arbitrarily small positive constant. Therefore, the Lebesgue convergence theorem in the Bochner integral immediately entails (\ref{EQ2}).

At the final step, we want to demonstrate the uniqueness of the global solution $u^*$ belonging to the space
\begin{align*}
 \mathcal{C}\big([0, \infty), L^2\big) \cap \mathcal{C}^1\big([0, \infty), H^s \cap L^\alpha\big). 
 \end{align*}
Indeed, let $u^*, v^*$ are solutions to (\ref{Main.Eq.1}) in this space. One recognizes that there exists a constant $M(T)$ fulfilling
$ \|\partial_t u^*\|_{X(T)}^{p-1}+ \|\partial_t v^*\|_{X(T)}^{p-1} \leq M(T). $ Following the same arguments as we did in the proof of Proposition \ref{Pro1.2} we obtain
 \begin{align*}
      \|\partial_t u^*-\partial_t v^*\|_{Y(t)} \lesssim M(T) \int_0^t \|\partial_t u^* -\partial_t v^*\|_{Y(\tau)}  d\tau
 \end{align*}
for all $t \in [0, T]$. From this, the
Gronwall inequality implies $ \partial_t u^* \equiv \partial_t v^*$ on $[0,T]$. Because $T $ is an arbitrary positive number, we claim that $\partial_t u^* \equiv \partial_t v^*$ on $[0, \infty)$. In addition, from the fact $u^*(0,x) = v^*(0,x) = u_0(x)$, this means that $u^* \equiv v^*$ on $[0, \infty)$. Hence, the proof of Theorem \ref{Theorem1} is established.

\section{Final remarks}
In this paper, we have succeeded in obtaining the global (in time) existence of small data Sobolev solutions to (\ref{Main.Eq.1}) in the one and two dimensional cases for all $p > 1$. Our aim in the future work is to study the following initial-boundary value problem for semi-linear damped wave
equations outside a closed unit ball, the so-called exterior domain $\Omega = \{x \in \mathbb{R}^n: |x| \ge 1\}$, with the power nonlinearity of derivative type $|u_t|^p$:
\begin{equation}\label{Eq4}
    \begin{cases}
       u_{tt}(t,x) -\Delta u(t,x) + u_t(t,x) = |u_t(t,x)|^p , &x \in \Omega, \, t > 0,\\
       u(t,x) = 0, &x \in \partial\Omega, \,t > 0,\\
       u(0,x) = \e u_0(x),\quad u_t(0,x) = \e u_1(x), &x \in \Omega.
    \end{cases}
\end{equation}
We can see that the authors in several previous works, for instance \cite{Ikeda2024, Ikehata2005, Ikehata2004, OgawaTakeda2009}, investigated \eqref{Eq4} with the usual power nonlinearity $|u|^p$ instead of $|u_t|^p$, and claimed that the critical exponent in $2$D is $p_{\rm crit} = 2 = p_{\rm Fuj}(2)$. This means that the critical exponents for semi-linear damped wave equations with the nonlinear term $|u|^p$ in both the whole space and the exterior domain really coincide at least in $2$D. From this fact, we conjecture that the same situation remains for \eqref{Eq4}, namely, we are able to show the global existence of small data Sobolev solutions for all $p > 1$ in some specific spatial dimensions. A positive answer for this problem will be partially given in our forthcoming work.



\appendix
\section{Some tools from Harmonic Analysis}

\begin{proposition}[Fractional Gagliardo-Nirenberg inequality, see \cite{Hajaiej2011}] \label{fractionalGagliardoNirenberg}
Let $1<q,\,q_1,\,q_2<\infty$, $s >0$ and $\theta\in [0,s)$. Then, it holds
\begin{align*}
 \|u\|_{\dot{H}^{\theta}_q} \lesssim \|u\|_{L^{q_1}}^{1-\omega(\theta,s)}\, \|u\|_{\dot{H}^{s}_{q_2}}^{\omega(\theta,s)}, 
 \end{align*}
where $\omega(\theta,s) =\dfrac{\frac{1}{q_1}-\frac{1}{q}+\frac{\theta}{n}}{\frac{1}{q_1}-\frac{1}{q_2}+\frac{s}{n}}$ and $\displaystyle\frac{\theta}{s}\leq \omega(\theta,s) \leq 1$.
\end{proposition}

\begin{proposition}[Fractional powers, see \cite{Duong2015}] \label{FractionalPowers}
Let $p>1$, $1< q <\infty$, where $s \in \big(n/q,p\big)$. Let us denote by $F(u)$ one of the functions $|u|^p,\, \pm |u|^{p-1}u$. Then, the following estimates hold:
\begin{align*}
\|F(u)\|_{H^{s}_q}\lesssim \|u\|_{H^{s}_q}\,\, \|u\|_{L^\infty}^{p-1} \quad \text{ and }\quad \| F(u)\|_{\dot{H}^{s}_q}\lesssim \|u\|_{\dot{H}^{s}_q}\,\, \|u\|_{L^\infty}^{p-1}. 
\end{align*}
\end{proposition}

\begin{proposition}[A fractional Sobolev embedding, see \cite{Dao2019}] \label{Embedding}
Let  $1 < q < \infty$ and $0< s_1< n/q < s_2$. Then, it holds
\begin{align*}
 \|u\|_{L^\ity} \lesssim \|u\|_{\dot{H}^{s_1}_q}+ \|u\|_{\dot{H}^{s_2}_q}. 
 \end{align*}
\end{proposition}

\begin{proposition}[Fractional chain rule, see \cite{Palmieri2018}]\label{chainrule}
    Let $s >0$, $p > \lceil s \rceil$ and $1 < q, q_1, q_2 < \infty$ satisfying the relation $1/q = (p-1)/q_1 + 1/q_2.$ Let us denote by $F(u)$ one of the functions $|u|^p,\, \pm |u|^{p-1}u$. Then, it holds
    \begin{align*}
        \| F(u)\|_{\dot{H}_q^s} \lesssim \|u\|_{L^{q_1}}^{p-1} \| u\|_{\dot{H}^s_{q_2}}.
    \end{align*}
\end{proposition}

\section*{Acknowledgements}
This work is supported by Vietnam Ministry of Education and Training
and Vietnam Institute for Advanced Study in Mathematics under grant number B2026-CTT-04.  The authors would like to thank Prof. Wenhui Chen (Guangzhou University) and Prof. Ryo Ikehata (Hiroshima University) for their helpful advice in the preparation of this
paper.

\end{document}